\documentclass[a4paper,11pt]{amsart}
\usepackage[utf8]{inputenc}
\usepackage[T2A]{fontenc}
\usepackage{amsmath,amssymb,amsthm}
\usepackage[a4paper,hmargin=2.5cm,vmargin=2.5cm]{geometry}
\usepackage{dsfont}
\usepackage{ytableau}
\usepackage{tikz}
\usepackage{rotating}
\usepackage{hyperref}
\usepackage[shortlabels]{enumitem}
\hypersetup{colorlinks=true, linkcolor=blue,citecolor=black!50!green}
\usepackage{todonotes}

\newenvironment{amssidewaysfigure}
  {\begin{sidewaysfigure}\vspace*{.5\textwidth}\begin{minipage}{\textheight}\centering}
  {\end{minipage}\end{sidewaysfigure}}

\newcommand{\CC}{{\mathbb{C}}}

\newcommand{\cI}{{\mathcal{I}}}

\newcommand{\cN}{{\mathcal{N}}}
\newcommand{\cO}{{\mathcal{O}}}

\newcommand{\cS}{{\mathcal{S}}}

\DeclareMathOperator{\GL}{GL}

\DeclareMathOperator{\Gr}{Gr}

\DeclareMathOperator{\Id}{Id}

\theoremstyle{plain}
\newtheorem{theorem}{Theorem}[section]
\newtheorem{lemma}[theorem]{Lemma}
\newtheorem{prop}[theorem]{Proposition}

\theoremstyle{definition}
\newtheorem{definition}[theorem]{Definition}
\newtheorem{example}[theorem]{Example}

\theoremstyle{remark}
\newtheorem{remark}[theorem]{Remark}

\newcommand{\Xc}{\mathcal X}

\newcommand{\gl}{\mathfrak{gl}}
\newcommand{\nf}{\mathfrak n}

\title{Partial order on involutive permutations and double Schubert cells}
\author{Evgeny Smirnov}
\email{esmirnov@hse.ru}
\address{HSE University, ul. Usacheva 6, 119048 Moscow, Russia}
\address{Independent University of Moscow, Bolshoi Vlassievskii per. 11, 119002 Moscow, Russia}
\address{Guangdong Technion -- Israel Institute of Technology, 214 Daxue rd, Shantou, Guangdong, 515063, China}
\date{\today}

\begin{document}

\begin{abstract}	 
	As shown by A.\,Melnikov, the orbits of a Borel subgroup acting by conjugation on upper-triangular matrices with square zero  are indexed by involutions in the symmetric group. The inclusion relation among the orbit closures defines a partial order on involutions. We observe that the same order on involutive permutations also arises  while describing the inclusion order on $B$-orbit closures in the direct product of two Grassmannians. We establish a geometric relation between these two settings.
\end{abstract}

\maketitle

\section{Introduction}

Let $\nf_n\subset \gl_n(\CC)$ be the Lie subalgebra of strictly upper-triangular matrices in the Lie algebra of complex $n\times n$ matrices. This subalgebra is equipped with the adjoint action of the standard (upper-triangular) Borel subgroup $B\subset \GL_n(\CC)$; this action has, generally speaking, infinitely many orbits. However, if we restrict this action to the set $\Xc_n\subset \nf_n$ of matrices with square zero, the adjoint action of $B$ on $\Xc_n$ has finitely many orbits. A.\,Melnikov~\cite{Mel00} has shown that these orbits are indexed by involutive permutations $\cI_{n}\subset \cS_n$. The inclusion of $B$-orbit closures on $\Xc_n$ defines a partial order on $\cI_n$, which is different from the Bruhat order. In her further paper~\cite{Mel06} Melnikov provides a simple combinatorial description of this order; another nice combinatorial interpretation was given by A.\,Knutson and P.\,Zinn-Justin in~\cite{KnutsonZinn14}. 

Quite unexpectedly, the same order appears in a different geometric setting. Consider the direct product of two Grassmannians $\Gr(k,n)\times\Gr(m,n)$ of $k$- and $m$-spaces in $\CC^n$. This variety is equipped with a componentwise action of the direct product of two Borel subgroups $B\times B\subset \GL(n)\times\GL(n)$, with its orbits being products of Schubert cells  $X^\circ_\lambda\times X^\circ_\mu$ in Grassmannians. One can also consider a finer orbit decomposition, provided by the diagonal Borel subgroup $B\subset B\times B$. It is well known (cf., for instance,~\cite{Lit94}) that the latter action also has finitely many orbits. Their explicit combinatorial description was obtained in~\cite{Smi08}. Moreover, the $B$-orbits constituting a given $(B\times B)$-orbit $X^\circ_\lambda\times X^\circ_\mu$ are indexed by a specific subset of involutive permutations $\cI_n(\lambda,\mu)\subset \cI_n$, depending upon  $\lambda$ and $\mu$. Like in the previous case, the inclusion of orbit closures defines a partial order on each subset of involutions $\cI_n(\lambda,\mu)$. It turns out that all these poset structures are inherited from the poset structure on $\cI_n$ defined by Melnikov. Our main result is the following theorem.

\begin{theorem}\label{thm:main} The partial order structure on each $\cI_n(\lambda,\mu)$ coming from the inclusion of $B$-orbit closures in the direct product of two Grassmannians is obtained by restricting of the adjoint partial order on $\cI_n$ to $\cI_n(\lambda,\mu)$.
\end{theorem}

This note is organized as follows. In Section~\ref{sec:prelim}, we recall the results of A.\,Melnikov on $B$-orbits in strictly triangular matrices with square zero; here we actively use the notation introduced by A.\,Knutson and P.\,Zinn-Justin. We also recall some basic facts on Schubert cells in Grassmannians. In Section~\ref{sec:grgr}, we give a combinatorial enumeration of $B$-orbits in a $(B\times B)$-orbit in the direct product of two Grassmannians and state the main result. The proof of the main result is given in Section~\ref{sec:proof}.

\section{Preliminaries}\label{sec:prelim}
	
\subsection{$B$-orbits in strictly triangular matrices with square zero}\label{sec:melnikov}

Consider the nilpotent orbit closure $\cN_n=\{X\in \gl_n(\CC)\mid X^2=0\}$. It is well known that this is a \emph{spherical variety}: the standard (upper-triangular) Borel subgroup $B\subset\GL(n)$ acts on $\cN$ by matrix conjugation with finitely many orbits. This set of orbits is a ranked poset, with the rank defined as the dimension of an orbit and the partial order defined by inclusion of orbit closures.

We will be interested not in the whole variety $\cN_n$, but rather in its intersection $\Xc_n=\cN_n\cap\nf_n$ with the set of strictly upper-triangular matrices. This situation was thoroughly studied by A.\,Melnikov \cite{Mel00,Mel06}. It turns out that the $B$-orbits in $\Xc_n$ are indexed by involutive permutations of $\{1,\dots,n\}$; we will denote the set of such permutations by $\cI_n=\{w\in \cS_n\mid w^2=\Id\}$.

\begin{theorem}[\cite{Mel00}] The set of $B$-orbits in $\Xc_n$ bijectively corresponds to the set of involutive permutations $\cI_n$. For each  orbit $\cO\subset \Xc_n$, there exists a unique permutation $w\in \cI_n$ such that $\cO=B\cdot w_<$. Here $w_<$ denotes the permutation matrix corresponding to $w$ with its diagonal and lower-triangular part replaced by zeros: $(w_<)_{ij}=1$ if $w(i)=j$ and $i<j$,	 and $(w_<)_{ij}=0$ otherwise.
\end{theorem}

Following the paper~\cite{KnutsonZinn14} by A.\,Knutson and P.\,Zinn-Justin, we will denote involutive permutations by arc diagrams. Namely, we draw nodes indexed by $1,\dots,n$ on a line and, if $w(i)=j$, join nodes $i$ and $j$ by an arc; if $w(i)=i$, we draw a vertical half-line from the node $i$, as shown in Figure~\ref{fig:arcexample} below.

\begin{figure}[h!]
\begin{tikzpicture}
		\draw (-0.5,0)--(7.5,0);
		\foreach \x in {1,...,8} \filldraw	(\x-1,0) circle [radius=0.1] node[below] {\footnotesize{\x}};
		\draw (1,0) to[out=90,in=90] (2,0);
		\draw (3,0)--(3,2);
		\draw (5,0)--(5,2);
		\draw (0,0) to[out=90,in=90] (6,0);
		\draw (4,0) to[out=90,in=90] (7,0);
\end{tikzpicture}	
\caption{Arc diagram corresponding to $w=\underline{73248615}=(17)(23)(58)$.}\label{fig:arcexample}
\end{figure}
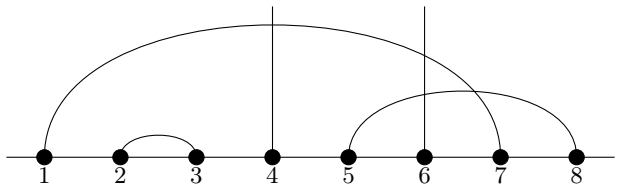

Such a presentation is very useful for computing the dimension of an orbit and describing the inclusion order on orbit closures. This is given by the following theorems.

\begin{theorem}[{\cite[\S2.7]{Mel06}}; {\cite[Theorem 4]{KnutsonZinn14}}]\label{thm:nilp-rank} Let $w\in\cI_n$. Then the dimension of the corresponding $B$-orbit $B\cdot w_<$ is equal to
\[
\dim B\cdot w_<=\#\text{arcs}\cdot(\#\text{arcs}+\#\text{half-lines})-\#\text{crossings}.
\]	
\end{theorem}

The maximal dimension of $B\cdot w_<$, equal to $\lfloor n^2/4\rfloor$, is achieved for crossingless arc diagrams with $\lfloor n/2\rfloor$ arcs. For $n\geq 3$, since the number of such arc diagrams is greater than one, the variety $\Xc_n$ is reducible (but equidimensional). Its irreducible components are called \emph{orbital varieties}.

The inclusion order on $B$-orbits also admits a nice description in terms of arc diagrams. Denote by $r_{ij}(w)$, with $i<j$, the number of pairs $(i',j')$ such that $i\leq i'<j'\leq j$ and $w(i')=j'$. Equivalently, this is the number of whole arcs in the interval $[i,j]$.

\begin{theorem}[{\cite[\S2.10]{Mel06}}; {\cite[Theorem 5]{KnutsonZinn14}}]\label{thm:nilp-rank2} For two involutions $v,w\in\cI_n$, we have $B\cdot v_<\subseteq \overline{B\cdot w_<}$ if and only if $r_{ij}(v)\leq r_{ij}(w)$ for each $i<j$.
\end{theorem}

This defines a partial order on the set of involutions: we shall say that $v\leq w$ if $B\cdot v_<\subseteq \overline{B\cdot w_<}$. 

\begin{remark}\label{rem:ranks} For an arbitrary element $X\in\Xc_n$, denote by $X_{ij}$ the submatrix formed by the rows $i,\dots,n$ and columns $1,\dots, j$. Suppose $X$ belongs to the orbit $B\cdot w_<$. Then $r_{ij}(w)$ for $i<j$ equals the rank of $X_{ij}$. Indeed, this is true for $X=w_<$, and the ranks of all $X_{ij}$ are constant along the $B$-orbits (they are invariant under the adjoint action of $B$).  Clearly, for $i\geq j$ the submatrices $X_{ij}$ are zero. 
\end{remark}

\begin{remark}
	This order on $\cI_n$ is different from the restriction of the Bruhat order on $\cS_n$ to $\cI_n$. In fact, the Bruhat order on $\cI_n$ corresponds to the inclusion of \emph{coadjoint} orbits, as opposed to the adjoint orbits considered here; for details, see~\cite{Ignatev12}.
\end{remark}

Figure~\ref{fig:main} represents the ranked poset $\cI_4$, with elements of the same rank (that is, with $B$-orbits of the same dimension) listed at the same horizontal level, from 4 (topmost) to 0.

\begin{figure}[h!]
\begin{tikzpicture}
	\begin{scope}[shift={(0,7)},scale=0.5]
		\draw (-0.5,0)--(3.5,0);
		\foreach \x in {1,...,4} \filldraw	(\x-1,0) circle [radius=0.1] node[below] {\footnotesize{\x}};
		\draw (0,0) to[out=90,in=90] (1,0);
		\draw (2	,0) to[out=90,in=90] (3,0);		
	\end{scope}
	\draw[dashed] (0.6,6.5)--(-1,5.6);
	\draw[dashed] (0.8,6.5)--(1.5,5.6);
	\draw[dashed] (1,6.5)--(7.5,5.6);
	\begin{scope}[shift={(5,7)}, scale=0.5]
		\draw (-0.5,0)--(3.5,0);
		\foreach \x in {1,...,4} \filldraw	(\x-1,0) circle [radius=0.1] node[below] {\footnotesize{\x}};
		\draw (0,0) to[out=90,in=90] (3,0);
		\draw (1,0) to[out=90,in=90] (2,0);
	\end{scope}
	\draw[dashed] (5.6,6.5)--(1.7,5.6);
	\draw[dashed] (5.8,6.5)--(4.7,5.6);
	\begin{scope}[shift={(7,5)}, scale=0.5]
		\draw (-0.5,0)--(3.5,0);
		\foreach \x in {1,...,4} \filldraw	(\x-1,0) circle [radius=0.1] node[below] {\footnotesize{\x}};
		\draw (0,0)--(0,1);
		\draw (1,0)--(1,1);
		\draw (2	,0) to[out=90,in=90] (3,0);		
	\end{scope}
	\draw[dashed] (7.6,4.5)--(6.2,3.6);
	\begin{scope}[shift={(1,5)}, scale=0.5]
		\draw (-0.5,0)--(3.5,0);
		\foreach \x in {1,...,4} \filldraw	(\x-1,0) circle [radius=0.1] node[below] {\footnotesize{\x}};
		\draw (0,0) to[out=90,in=90] (2,0);		
		\draw (1,0) to[out=90,in=90] (3,0);		
	\end{scope}
	\draw[dashed] (2,4.5)--(5.8,3.6);
	\draw[dashed] (1.8,4.5)--(0.8,3.6);
	\begin{scope}[shift={(4,5)}, scale=0.5]
		\draw (-0.5,0)--(3.5,0);
		\foreach \x in {1,...,4} \filldraw	(\x-1,0) circle [radius=0.1] node[below] {\footnotesize{\x}};
		\draw (0,0)--(0,1);
		\draw (1,0) to[out=90,in=90] (2,0);		
		\draw (3,0)--(3,1);		
	\end{scope}
	\draw[dashed] (5,4.5)--(6,3.6);
	\draw[dashed] (4.8,4.5)--(1,3.6);
	\begin{scope}[shift={(-2,5)}, scale=0.5]
		\draw (-0.5,0)--(3.5,0);
		\foreach \x in {1,...,4} \filldraw	(\x-1,0) circle [radius=0.1] node[below] {\footnotesize{\x}};
		\draw (2,0)--(2,1);
		\draw (3,0)--(3,1);
		\draw (0,0) to[out=90,in=90] (1,0);		
	\end{scope}	
	\draw[dashed] (-1.2,4.5)--(0.6,3.6);
	\begin{scope}[shift={(0,3)}, scale=0.5]
		\draw (-0.5,0)--(3.5,0);
		\foreach \x in {1,...,4} \filldraw	(\x-1,0) circle [radius=0.1] node[below] {\footnotesize{\x}};
		\draw (1,0)--(1,1);
		\draw (3,0)--(3,1);
		\draw (0,0) to[out=90,in=90] (2,0);		
	\end{scope}
	\draw[dashed] (0.8,2.5)--(3.6,1.6);
	\begin{scope}[shift={(5,3)}, scale=0.5]
		\draw (-0.5,0)--(3.5,0);
		\foreach \x in {1,...,4} \filldraw	(\x-1,0) circle [radius=0.1] node[below] {\footnotesize{\x}};
		\draw (2,0)--(2,1);
		\draw (0,0)--(0,1);
		\draw (1,0) to[out=90,in=90] (3,0);		
	\end{scope}
	\draw[dashed] (5.6,2.5)--(3.8,1.6);
	\begin{scope}[shift={(3,1)}, scale=0.5]
		\draw (-0.5,0)--(3.5,0);
		\foreach \x in {1,...,4} \filldraw	(\x-1,0) circle [radius=0.1] node[below] {\footnotesize{\x}};
		\draw (2,0)--(2,1);
		\draw (1,0)--(1,1);
		\draw (0,0) to[out=90,in=90] (3,0);		
	\end{scope}
	\draw[dashed] (3.8,0.5)--(3.8,-0.4);
	\begin{scope}[shift={(3,-1)}, scale=0.5]
		\draw (-0.5,0)--(3.5,0);
		\foreach \x in {1,...,4} \filldraw	(\x-1,0) circle [radius=0.1] node[below] {\footnotesize{\x}};
		\draw (0,0)--(0,1);
		\draw (2,0)--(2,1);
		\draw (1,0)--(1,1);
		\draw (3,0)--(3,1);
	\end{scope}
\end{tikzpicture}
\caption{Inclusion order on  arc diagrams for $n=4$}\label{fig:main}
\end{figure}
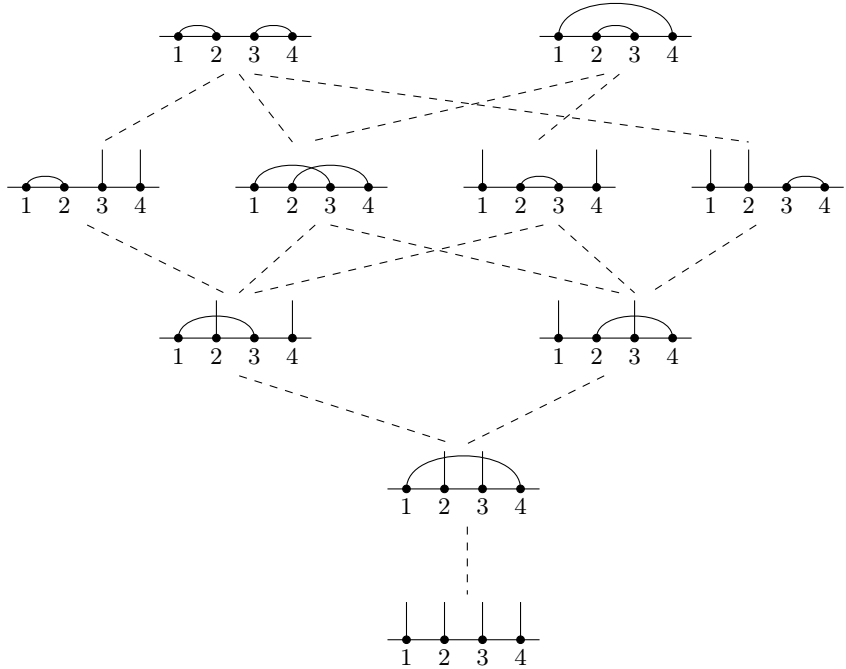	

\subsection{$B$-orbits in Grassmannians} This subsection is devoted to fixing the notation and describing $B$-orbits in one Grassmannian. As before, we let $B
\subset \GL(n)$ be the subgroup of nondegenerate upper-triangular matrices. We also fix the maximal torus $T\subset B$;  it consists of nondegenerate diagonal matrices.

Denote by $\Gr(k,n)$ the Grassmannian of $k$-dimensional vector subspaces in an $n$-dimensional vector space. It is a $\GL(n)$-homogeneous space, with finitely many (namely, $\binom{n}{k}$) orbits of a Borel subgroup $B$. These orbits are indexed by Young diagrams $\lambda=(\lambda_1,\dots,\lambda_k)$ with at most $k$ parts not exceeding $n-k$.

This parametrization is as follows. Each $B$-orbit contains a unique $T$-stable point. If $e_1,\dots,e_n$ denotes the standard basis of $\CC^n$, then the $T$-stable points correspond to subspaces spanned by $k$ basis vectors. To a Young diagram $\lambda$ we assign the following subspace:
\[
U_\lambda=\langle e_{\lambda_k+1}, e_{\lambda_{k-1}+2},\dots,e_{\lambda_1+k}\rangle.
\]
The $B$-orbits, usually called \emph{Schubert cells}, will be further denoted by $X^\circ_{\lambda}=B\cdot U_\lambda$. It is well known that $X_\lambda^\circ$ is isomorphic to an affine space of dimension $|\lambda|$ and that $X_\lambda^\circ\subseteq \overline{X_\mu^\circ}$ if and only if $\lambda\subseteq\mu$; see~\cite{Manivel98}, \cite{Fulton97} or any other textbook on this topic for details.

\section{$B$-orbits in double Grassmannians}\label{sec:grgr} It is well known (see, for example, \cite{Lit94,MWZ99}) that the direct product of two Grassmannians $\Gr(k,n)\times\Gr(m,n)$ is a spherical variety with respect to the action of the diagonal subgroup $B\subset B\times B$. In geometric terms, this means that the number of triples consisting of a $k$-plane, an $m$-plane, and a full flag in $\CC^n$, considered up to $\GL(n)$-action, is finite.

\subsection{Combinatorial description of orbits}
Here we recall the combinatorial description of $B$-orbits acting on the direct product of two Grassmannians $X=\Gr(k,n)\times\Gr(m,n)$. We assume $k$, $m$, and $n$ to be fixed throughout this section. This description appeared in a slightly different form in our paper~\cite{Smi08}. It also follows from much more general results by P.\,Magyar, J.\,Weyman, and A.\,Zelevinsky, see~\cite{MWZ99}.

The $(B\times B)$-orbits in $X$ are indexed by pairs of Young diagrams $\lambda,\mu$, where $\lambda\subseteq k\times (n-k)$ and $\mu\subset m\times (n-m)$ are partitions with at most $k$ (resp. $m$) parts not exceeding $n-k$ (resp. $n-m$).   Each of these orbits is the direct product of two Schubert cells $X^\circ_\lambda\times X^\circ_\mu$. 

Given partitions $\lambda$ and $\mu$ in a rectangle of semiperimeter $n$, we can assign to them \emph{bit strings} (sequences of zeroes and ones) $s(\lambda),s(\mu)\in\{0,1\}^n$ of length $n$ as follows. For $\lambda$, let $s_i(\lambda)=1$ if $i$ occurs among the numbers $\lambda_k+1,\lambda_{k-1}+2,\dots,\lambda_1+k$, and $0$ otherwise. Graphically this can be interpreted as follows: Young diagram $\lambda$ is bounded from below by a lattice path of length $n$, going from the southwestern corner to the northeastern one. The number $s_i(\lambda)$ is equal to $1$ if $i$-th segment is vertical, and to $0$ if it is horizontal. Similarly we define the bit string $s(\mu)$.

Our next goal is to define a subset $\cI_n(\lambda,\mu)$ of the set of involutive permutations $\cI_n\subset S_n$. Take the componentwise sum $s(\lambda,\mu)=s(\lambda)+s(\mu)\in \{0,1,2\}^n$. This is an $n$-tuple consisting of zeroes, ones, and twos. Its $i$-th component will be denoted by $s_i(\lambda,\mu)$.

\begin{definition} An involutive permutation $w\in \cI_n$, $w^2=\Id$, is said to be \emph{consistent with the pair $(\lambda,\mu)$} (or just \emph{$(\lambda,\mu)$-consistent}) if for every pair $(i,j)$, $1\leq i<j\leq n$, such that $w(i)=j$, $w(j)=i$, we have $s_i(\lambda,\mu)=0$ and $s_j(\lambda,\mu)=2$.	 The set of all such permutations is denoted by $\cI_n(\lambda,\mu)$.
\end{definition}

Informally, this means that for each transposition $(i,j)$ occurring in $w$, with $i<j$, the $i$-th segments of the lattice paths defined by both $\lambda$ and $\mu$ are horizontal, while the $j$-th segments of both paths are vertical. This means that the involutions in $\cI_n(\lambda,\mu)$ have prescribed sets of possible ``left endpoints'' and ``right endpoints'', not necessarily of the same cardinality. 

\begin{example}\label{ex:22square} Let $k=m=2$, $n=4$, and $\lambda=\mu=(2,2)$. Then the set of involutions $\cI_4(\lambda,\mu)$ consistent with these partitions has seven elements:
\[
\Id,\qquad (13), \qquad (14),\qquad (23),\qquad (24),\qquad (13)(24),\qquad (14)(23).
\]	
\end{example}

\begin{example}\label{ex:empty} For certain pairs $\lambda$, $\mu$, the set $\cI_n(\lambda,\mu)$ can consist only of the identity permutation. For example, take $k=m=2$, $n=4$, and $\lambda=(2,1)$, $\mu=(1,0)$. Then $s(\lambda,\mu)=(1,1,1,1)$. Another example with $\cI_n(\lambda,\mu)=\{\Id\}$ is given by $\lambda=(1,0)$, $\mu=(0,0)$. In this case, $s(\lambda,\mu)=(2,1,1,0)$. Since this sequence does not contain $2$'s preceded by $0$'s, any permutation consistent with cannot contain a nontrivial transposition.
\end{example}

Similarly to the previous section, we shall denote involutive permutations by arc diagrams. Let us place $n$ nodes on a line, numbered $1,\dots,n$ from left to right. We will say that $i$-th node is \emph{black} if $s_i(\lambda,\mu)=0$, \emph{white} if $s_i(\lambda,\mu)=2$, and \emph{grey} if $s_i(\lambda,\mu)=1$. Given an involutive permutation, we draw an arc in the upper half-plane joining each pair $(i,j)$ such that $w(i)=j$. Moreover, let us draw vertical half-lines going up from all black and white (but not grey) vertices corresponding to fixed points of $w$. 
An involutive permutation represented in such a form is $(\lambda,\mu)$-consistent if the left end of each arc is black and the right end is white. 

This arc interpretation allows us to define a number $d(w)=d(w,\lambda,\mu)$ for each $w\in\cI_n(\lambda,\mu)$. Define it as follows:
\[
d(w)=\#\{\text{crossings in the arc diagram}\}+\#\{(i,j)\mid i<j, w(i)=i,w(j)=j, s_i=0, s_j=2\}.
\]
The second summand is the number of pairs consisting of a black vertex $i$ and a white vertex $j$ with vertical lines going from them, such that $i<j$. Informally, these two vertical lines can be thought of as ``crossing at infinity''. Note that $d(w)$ depends not only on $w$, but also on $\lambda$ and $\mu$.

\begin{example}\label{ex:arcdiag} Let $k=4$, $m=5$, $n=9$. Consider two Young diagrams $\lambda=(5,4,2,1)$ and $\mu=(4,4,4,1,1)$. Then we have
\[
s(\lambda)=(0,1,	0,1,0,0,1,0,1);\quad s(\mu)=(0,1,1,0,0,0,1,1,1);\quad s(\lambda,\mu)=(0,2,1,1,0,0,2,1,2).
\]
In Figure~\ref{fig:arcexample2} we give the arc diagram of permutation $w=(17)(59)\in\cI_n(\lambda,\mu)$. For this permutation, $d(w,\lambda,\mu)=4$. Indeed, there are four crossings and no pairs of black and white half-lines: note that 2 and 6 do not form such a pair, since the white vertex precedes the black one.
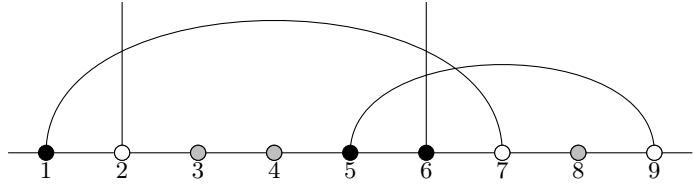
\begin{figure}[h!]
\begin{tikzpicture}
		\draw (-0.5,0)--(8.5,0);
		\foreach \x in {1,...,9} \filldraw	(\x-1,0) circle [radius=0.1] node[below] {\footnotesize{\x}};
		\draw (1,0)--(1,2);

		\draw (5,0)--(5,2);
		\draw (0,0) to[out=90,in=90] (6,0);
		\draw (4,0) to[out=90,in=90] (8,0);
		\foreach \x in {2,7,9} \filldraw[fill=white]	(\x-1,0) circle [radius=0.1];
		\foreach \x in {3,4,8} \filldraw[fill=black!25]	(\x-1,0) circle [radius=0.1];
\end{tikzpicture}	
\caption{Arc diagram corresponding to $w=\underline{723496185}=(17)(59)$.}\label{fig:arcexample2}
\end{figure}

\end{example}

These invariants are essential for describing the inclusion order on $B$-orbits inside the $(B\times B)$-orbit $X_\lambda^\circ\times X_\mu^\circ$.

\begin{theorem}[\cite{Smi08}]\label{thm:smi08}\begin{enumerate}[(i)]\item\label{thm:smi08i} Orbits of the Borel subgroup $B$ inside the $(B\times B)$-orbit $X_\lambda^\circ\times X_\mu^\circ\subset\Gr(k,n)\times \Gr(m,n)$ bijectively correspond to the elements of $\cI_n(\lambda,\mu)$;

	\item\label{thm:smi08ii} Let $e_1,\dots, e_n$ be a basis of $\CC^n$ that agrees with the choice of $B\subset\GL(n)$. Then each orbit $\cO_{\lambda\mu}^w$ is obtained as the $B$-orbit of the  pair of subspaces $(U,W)$, where: 
	\begin{eqnarray*}
		U&=&\left\langle e_{j}\mid s_j(\lambda)=1, w(j)=j\right\rangle+\langle e_{w(j)}+e_{j}\mid s_j(\lambda)=1, w(j)\neq j\rangle,\\ 
		W&=&\langle e_{\mu_1},\dots, e_{\mu_m}\rangle.
	\end{eqnarray*}
	
	\item\label{thm:smi08iii}  the codimension of orbit $\cO^w_{\lambda\mu}\subseteq X_\lambda^\circ\times X_\mu^\circ$ equals $d(w,\lambda,\mu)$.
\end{enumerate} 
\end{theorem}

For instance, the ``canonical'' representative of the orbit given by $\lambda$, $\mu$, and $w$ from Example~\ref{ex:arcdiag} is as follows:
\[
U=\langle e_2, e_4, e_1+e_7, e_5+e_9\rangle, \qquad W=\langle e_2,e_3,e_7,e_8,e_9\rangle.
\]


Our next observation is as follows.

\begin{prop} There exists a unique maximal and unique minimal $B$-orbit inside $X_\lambda^\circ\times X_\mu^\circ$.	
\end{prop}

\begin{proof} Existence of a maximal (open) orbit $\cO^{\max}_{\lambda\mu}$ is obvious, since $X_\lambda^\circ\times X_\mu^{\circ}$ is irreducible (it is isomorphic to affine space $\CC^{|\lambda|+|\mu|}$). The corresponding arc diagram is obtained as follows: given a black-white-grey coloring of $\{1,\dots,n\}$, join by an arc a black and a white vertex with possibly only grey vertices between them. Repeat this procedure (ignoring vertices with arcs) until there are no more black-white pairs left. All the remaining black and white vertices, white ones coming before black ones, are joined with infinity. Obviously, such a matching is crossingless.

The minimal orbit $\cO_{\lambda\mu}^{\min} =\cO_{\lambda\mu}^{\Id}$ is easier to construct: it corresponds to the identity permutation $\Id\in \cI_n(\lambda,\mu)$, with all the black and white vertices defining vertical half-lines. Its codimension is equal to the number of pairs consisting of a black and a white vertex, in this order from left to right. Note that this orbit contains a unique $(T\times T)$-stable point: it is exactly the point given in~Theorem~\ref{thm:smi08}, part \ref{thm:smi08ii}.
\end{proof}

\subsection{Inclusion order on $B$-orbit closures}
 In this subsection we give a description of the inclusion order of orbit closures in $X_\lambda^\circ\times X_\mu^\circ$ in terms of ranks.

Namely, for each pair $(i,j)$ consisting of a black and a white vertex (in this order), we define
\[
r_{ij}(w,\lambda,\mu)=\#\{(i',j')\mid w(i')=j', i\leq i'<j\leq j'\}.
\]
In other words, $r_{ij}(w,\lambda,\mu)$ is the number of arcs situated inside the interval $[i,j]$. Then inclusion of $B$-orbit closures inside a $(B\times B)$-orbit is given by inequalities of ranks.

\begin{theorem}\label{thm:rank}
	For any $v,w\in\cI_n(\lambda,\mu)$, we have $\cO_{\lambda\mu}^v \subseteq \overline{\cO_{\lambda\mu}^w}$ if and only if $r_{ij}(w,\lambda,\mu)\geq r_{ij}(v,\lambda,\mu)$ for each $i<j$ with $s_i(\lambda,\mu)=0$ and $s_j(\lambda,\mu)=2$.
\end{theorem}

This theorem looks almost the same as Theorem~\ref{thm:nilp-rank2}. It immediately implies Theorem~\ref{thm:main}.


It is not hard to prove Theorem~\ref{thm:rank} similarly to the proof of~Theorem~\ref{thm:nilp-rank2} given in~\cite{KnutsonZinn14}: describe covering relations in the poset $\cI_n{(\lambda,\mu)}$ explicitly (this was done in~\cite{Smi08}), then for each covering relation construct an explicit degeneration of the larger orbit to the smaller one, like in Proposition~1 of~\cite{KnutsonZinn14}, and then use semicontinuity of ranks.  Then Theorem~\ref{thm:main} follows from an \emph{a posteriori} comparison of Theorem~\ref{thm:rank} with Theorem~\ref{thm:nilp-rank2}. However, this does not fully explain this ``partial order restriction phenomenon''.

We will use a different approach, a more geometric one. For this we construct a slice  $S\subset X_\lambda^\circ\times X_\mu^\circ$ which intersects all $B$-orbits transversally and  has dimension complementary to $\dim\cO_{\lambda\mu}^{\min}$. Then this slice can be  embedded into the space of upper-triangular matrices with square zero. This will be done in Section~\ref{sec:proof}.

Recall that for a poset $(M,\leq)$ a relation $a\leq b$ is said to be \emph{covering} if for any $c\in M$ such that $a\leq c\leq b$ we have either $a=c$ or $c=b$ (that is, there are no intermediate elements between $a$ and $b$).

\begin{remark}\label{rem:nocover}
	The embedding of posets $\cI_n(\lambda,\mu)\hookrightarrow \cI_n$ does not preserve covering relations.
\end{remark}

An example of such a situation is as follows. Let $n=4$; the poset structure of $\cI_4$ given by  $B$-orbits on $\Xc_4$ is shown on Figure~\ref{fig:main}. Now take $\lambda=\mu=(2,1)$. The set of involutions consistent with $s(\lambda,\mu)=(0,2,0,2)$ has five elements; they are shown on Figure~\ref{fig:twopics} (right). In this order element $(12)$ covers $(14)$, while in the order on $\cI_4$ there is an intermediate element $(13)$ between them; this element does not belong to $\cI_4(\lambda,\mu)$.

\begin{amssidewaysfigure}
\begin{tikzpicture}[scale=0.9]
	\begin{scope}[shift={(0,7)},scale=0.5, color=black!20]
		\draw (-0.5,0)--(3.5,0);
		\foreach \x in {1,...,4} \filldraw	(\x-1,0) circle [radius=0.1]; 
		\draw (0,0) to[out=90,in=90] (1,0);
		\draw (2	,0) to[out=90,in=90] (3,0);		
	\end{scope}
	\begin{scope}[shift={(5,7)}, scale=0.5]
		\draw (-0.5,0)--(3.5,0);
		\foreach \x in {1,...,4} \filldraw	(\x-1,0) circle [radius=0.1] node[below] {\footnotesize{\x}};
		\draw (0,0) to[out=90,in=90] (3,0);
		\draw (1,0) to[out=90,in=90] (2,0);
		\foreach \x in {3,4} \filldraw[fill=white]	(\x-1,0) circle [radius=0.1];	
	\end{scope}
	\draw[dashed] (5.6,6.5)--(1.7,5.6);
	\draw[dashed] (5.8,6.5)--(4.7,5.6);
	\begin{scope}[shift={(7,5)}, scale=0.5,color=black!20]
		\draw (-0.5,0)--(3.5,0);
		\foreach \x in {1,...,4} \filldraw	(\x-1,0) circle [radius=0.1];
		\draw (0,0)--(0,1);
		\draw (1,0)--(1,1);
		\draw (2	,0) to[out=90,in=90] (3,0);		
	\end{scope}
	\begin{scope}[shift={(1,5)}, scale=0.5]
		\draw (-0.5,0)--(3.5,0);
		\foreach \x in {1,...,4} \filldraw	(\x-1,0) circle [radius=0.1] node[below] {\footnotesize{\x}};
		\draw (0,0) to[out=90,in=90] (2,0);		
		\draw (1,0) to[out=90,in=90] (3,0);	
		\foreach \x in {3,4} \filldraw[fill=white]	(\x-1,0) circle [radius=0.1];
	\end{scope}
	\draw[dashed] (2,4.5)--(5.8,3.6);
	\draw[dashed] (1.8,4.5)--(0.8,3.6);
	\begin{scope}[shift={(4,5)}, scale=0.5]
		\draw (-0.5,0)--(3.5,0);
		\foreach \x in {1,...,4} \filldraw	(\x-1,0) circle [radius=0.1] node[below] {\footnotesize{\x}};
		\draw (0,0)--(0,1);
		\draw (1,0) to[out=90,in=90] (2,0);		
		\draw (3,0)--(3,1);		
		\foreach \x in {3,4} \filldraw[fill=white]	(\x-1,0) circle [radius=0.1];
	\end{scope}
	\draw[dashed] (5,4.5)--(6,3.6);
	\draw[dashed] (4.8,4.5)--(1,3.6);
	\begin{scope}[shift={(-2,5)}, scale=0.5, color=black!20]
		\draw (-0.5,0)--(3.5,0);
		\foreach \x in {1,...,4} \filldraw	(\x-1,0) circle [radius=0.1];
		\draw (2,0)--(2,1);
		\draw (3,0)--(3,1);
		\draw (0,0) to[out=90,in=90] (1,0);		
	\end{scope}	
	\begin{scope}[shift={(0,3)}, scale=0.5]
		\draw (-0.5,0)--(3.5,0);
		\foreach \x in {1,...,4} \filldraw	(\x-1,0) circle [radius=0.1] node[below] {\footnotesize{\x}};
		\draw (1,0)--(1,1);
		\draw (3,0)--(3,1);
		\draw (0,0) to[out=90,in=90] (2,0);		
		\foreach \x in {3,4} \filldraw[fill=white]	(\x-1,0) circle [radius=0.1];
	\end{scope}
	\draw[dashed] (0.8,2.5)--(3.6,1.6);
	\begin{scope}[shift={(5,3)}, scale=0.5]
		\draw (-0.5,0)--(3.5,0);
		\foreach \x in {1,...,4} \filldraw	(\x-1,0) circle [radius=0.1] node[below] {\footnotesize{\x}};
		\draw (2,0)--(2,1);
		\draw (0,0)--(0,1);
		\draw (1,0) to[out=90,in=90] (3,0);		
		\foreach \x in {3,4} \filldraw[fill=white]	(\x-1,0) circle [radius=0.1];	
	\end{scope}
	\draw[dashed] (5.6,2.5)--(3.8,1.6);
	\begin{scope}[shift={(3,1)}, scale=0.5]
		\draw (-0.5,0)--(3.5,0);
		\foreach \x in {1,...,4} \filldraw	(\x-1,0) circle [radius=0.1] node[below] {\footnotesize{\x}};
		\draw (2,0)--(2,1);
		\draw (1,0)--(1,1);
		\draw (0,0) to[out=90,in=90] (3,0);		
		\foreach \x in {3,4} \filldraw[fill=white]	(\x-1,0) circle [radius=0.1];
	\end{scope}
	\draw[dashed] (3.8,0.5)--(3.8,-0.4);
	\begin{scope}[shift={(3,-1)}, scale=0.5]
		\draw (-0.5,0)--(3.5,0);
		\foreach \x in {1,...,4} \filldraw	(\x-1,0) circle [radius=0.1] node[below] {\footnotesize{\x}};
		\draw (0,0)--(0,1);
		\draw (2,0)--(2,1);
		\draw (1,0)--(1,1);
		\draw (3,0)--(3,1);
		\foreach \x in {3,4} \filldraw[fill=white]	(\x-1,0) circle [radius=0.1];
	\end{scope}
\end{tikzpicture}
\qquad
\begin{tikzpicture}
	\begin{scope}[shift={(0,7)},scale=0.5]
		\draw (-0.5,0)--(3.5,0);
		\foreach \x in {1,...,4} \filldraw	(\x-1,0) circle [radius=0.1] node[below] {\footnotesize{\x}};
		\draw (0,0) to[out=90,in=90] (1,0);
		\draw (2	,0) to[out=90,in=90] (3,0);
		\foreach \x in {2,4} \filldraw[fill=white]	(\x-1,0) circle [radius=0.1];
	\end{scope}
	\draw[dashed] (0.6,6.5)--(-1,5.6);
	\draw[dashed] (1,6.5)--(7.5,5.6);
	\begin{scope}[shift={(5,7)}, scale=0.5,color=black!20]
		\draw (-0.5,0)--(3.5,0);
		\foreach \x in {1,...,4} \filldraw	(\x-1,0) circle [radius=0.1];
		\draw (0,0) to[out=90,in=90] (3,0);
		\draw (1,0) to[out=90,in=90] (2,0);
		\foreach \x in {3,4} \filldraw[fill=white]	(\x-1,0) circle [radius=0.1];	
	\end{scope}
	\begin{scope}[shift={(7,5)}, scale=0.5]
		\draw (-0.5,0)--(3.5,0);
		\foreach \x in {1,...,4} \filldraw	(\x-1,0) circle [radius=0.1] node[below] {\footnotesize{\x}};
		\draw (0,0)--(0,1);
		\draw (1,0)--(1,1);
		\draw (2	,0) to[out=90,in=90] (3,0);
		\foreach \x in {2,4} \filldraw[fill=white]	(\x-1,0) circle [radius=0.1];	
	\end{scope}
	\draw[dashed] (7.6,4.5)--(3.8,1.6);
	\begin{scope}[shift={(1,5)}, scale=0.5,color=black!20]
		\draw (-0.5,0)--(3.5,0);
		\foreach \x in {1,...,4} \filldraw	(\x-1,0) circle [radius=0.1];
		\draw (0,0) to[out=90,in=90] (2,0);		
		\draw (1,0) to[out=90,in=90] (3,0);	
		\foreach \x in {3,4} \filldraw[fill=white]	(\x-1,0) circle [radius=0.1];
	\end{scope}
	\begin{scope}[shift={(4,5)}, scale=0.5,color=black!20]
		\draw (-0.5,0)--(3.5,0);
		\foreach \x in {1,...,4} \filldraw	(\x-1,0) circle [radius=0.1];
		\draw (0,0)--(0,1);
		\draw (1,0) to[out=90,in=90] (2,0);		
		\draw (3,0)--(3,1);		
		\foreach \x in {3,4} \filldraw[fill=white]	(\x-1,0) circle [radius=0.1];
	\end{scope}
	\draw[dashed] (-1.2,4.5)--(3.6,1.6);
	\begin{scope}[shift={(0,3)}, scale=0.5,color=black!20]
		\draw (-0.5,0)--(3.5,0);
		\foreach \x in {1,...,4} \filldraw	(\x-1,0) circle [radius=0.1];
		\draw (1,0)--(1,1);
		\draw (3,0)--(3,1);
		\draw (0,0) to[out=90,in=90] (2,0);		
		\foreach \x in {3,4} \filldraw[fill=white]	(\x-1,0) circle [radius=0.1];
	\end{scope}
	\begin{scope}[shift={(5,3)}, scale=0.5,color=black!20]
		\draw (-0.5,0)--(3.5,0);
		\foreach \x in {1,...,4} \filldraw	(\x-1,0) circle [radius=0.1];
		\draw (2,0)--(2,1);
		\draw (0,0)--(0,1);
		\draw (1,0) to[out=90,in=90] (3,0);		
		\foreach \x in {3,4} \filldraw[fill=white]	(\x-1,0) circle [radius=0.1];	
	\end{scope}
	\begin{scope}[shift={(3,1)}, scale=0.5]
		\draw (-0.5,0)--(3.5,0);
		\foreach \x in {1,...,4} \filldraw	(\x-1,0) circle [radius=0.1] node[below] {\footnotesize{\x}};
		\draw (2,0)--(2,1);
		\draw (1,0)--(1,1);
		\draw (0,0) to[out=90,in=90] (3,0);		
		\foreach \x in {2,4} \filldraw[fill=white]	(\x-1,0) circle [radius=0.1];
	\end{scope}
	\draw[dashed] (3.8,0.5)--(3.8,-0.4);
	\begin{scope}[shift={(3,-1)}, scale=0.5]
		\draw (-0.5,0)--(3.5,0);
		\foreach \x in {1,...,4} \filldraw	(\x-1,0) circle [radius=0.1] node[below] {\footnotesize{\x}};
		\draw (0,0)--(0,1);
		\draw (2,0)--(2,1);
		\draw (1,0)--(1,1);
		\draw (3,0)--(3,1);
		\foreach \x in {2,4} \filldraw[fill=white]	(\x-1,0) circle [radius=0.1];
	\end{scope}
\end{tikzpicture}
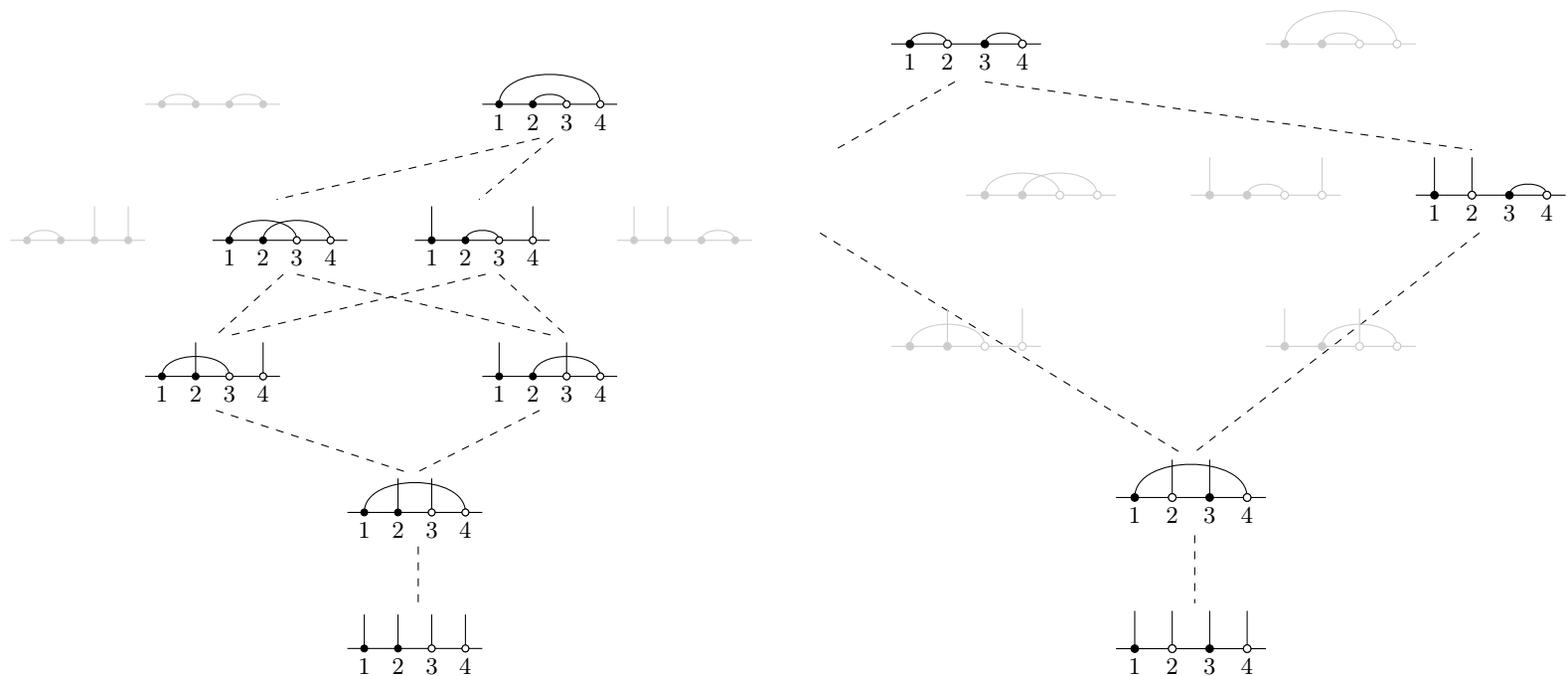
\caption{Inclusion order on $\cI_4(\lambda,\mu)$ for $\lambda=\mu=(2,2)$ (left) and $\lambda=\mu=(2,1)$ (right)}\label{fig:twopics}
\end{amssidewaysfigure}

\section{Proof of the main result}\label{sec:proof}

The main result is immediately implied by the following lemma. 

\begin{lemma}\label{lem:slice} There exists a subvariety (slice) $S\subset X_\lambda^\circ\times X_\mu^\circ$ such that:
\begin{enumerate}[(i)]
	\item $S$ is isomorphic to an affine space of dimension $\dim S=d(\Id,\lambda,\mu)$. That is, $\dim S$ equals the codimension of $\cO_{\lambda\mu}^{\min}$;
	\item $S$ intersects each orbit closure $\overline{\cO_{\lambda\mu}^w}$ transversally; in particular, $S\cap \cO^{\min}_{\lambda\mu}=\{\mathrm{pt}\}$;
	\item there exists an embedding $\imath \colon S\hookrightarrow \Xc_n$ such that for each $w\in\cI_n(\lambda,\mu)$ we have 
	\[
	\imath(\cO^w_{\lambda\mu}\cap  S)\subseteq B\cdot w_<.
	\]
\end{enumerate}
\end{lemma}

\begin{proof} We construct this slice explicitly. It will consist of a pair of two subspaces $(U(t_{ij}),W)$, where $W=\langle e_{\mu_m+1}, e_{\mu_{m-1}+2},\dots,e_{\mu_1+m}\rangle$ is fixed, and $U$ is spanned by vectors $e_{\lambda_i+k-i}$, where $s_{\lambda_i+k-i}(\lambda,\mu)=1$ (that is, the corresponding vertex of the arc diagram is grey), and by vectors
\[
e_j+\sum_{i<j, s_i(\lambda,\mu)=0} t_{ij}e_i,
\]
where $j$ runs over the set of white vertices (i.e., $s_j(\lambda,\mu)=2$).

It is clear that all such subspaces are in the $(B\times B)$-orbit $X_\lambda^\circ\times X_\mu^\circ$. The number of parameters defining $S$ is equal to the number of pairs $(i,j)$, with $i<j$, $s_i(\lambda,\mu)=0$, and $s_j(\lambda,\mu)=2$, that is, to $d(\Id,\lambda,\mu)$. 

The slice $S$ contains  the representative of each of the orbits $\cO^w_{\lambda\mu}$ (see Theorem~\ref{thm:smi08}, \ref{thm:smi08ii}). Indeed, given $w\in\cI_n(\lambda,\mu)$, take $t_{ij}=1$ if $w(i)=j$ and $i<j$, and $t_{ij}=0$ otherwise. Transversality of $S$ and $\cO^w_{\lambda\mu}$  is shown by direct computation.

The final part is to construct the embedding of $S$ into $\Xc_n$. This is done in the most obvious way: a point of the slice defined by parameters $(t_{ij})$ goes to matrix with its $(i,j)$-th element equal to $t_{ij}$. Since the sets of all possible $i$'s (black vertices) and $j$'s (white vertices) are disjoint, the square of such a matrix is zero.

Let us  prove the last assertion: $\imath({\cO^w_{\lambda\mu}}\cap S)\subseteq \overline{B\cdot w_<}$. Indeed, the ``canonical'' pair of subspaces $(U,W)\subset \cO_{\lambda\mu}^w$ is mapped exactly into $w_<$. Moreover, every element of \mbox{$\cO^w_{\lambda\mu}\cap S$} has the form $(b\cdot U,W)$ for some element $b\in B$. So the matrix $X=\imath((b\cdot U,W))\in \Xc_n$ is obtained from $w_<$ by adding linear combinations of columns with smaller numbers to columns with bigger numbers and linear combinations of rows with larger numbers to rows with smaller numbers. These operations do not change the ranks of the southwestern submatrices $X_{ij}$, so all such elements are contained in the same orbit $B\cdot w_<$.
\end{proof}

\begin{example}
	Consider this construction for $\lambda=(4,4,2)$ and $\mu=(3,3,1,1)$. In this case, $s(\lambda,\mu)=(0,1,2,0,0,2,2)$. Vertices 1, 4, 5 of the arc diagram are black, vertices 3, 6, 7 are white, and vertex 2 is grey. Then $S$ is 7-dimensional, and the corresponding subspaces look as follows:
	\[
	U(t_{ij})=\langle e_3+t_{13}e_1, e_6+t_{16}e_1+t_{46}e_4, e_7+t_{17}e_1+t_{47}e_7\rangle,\qquad W=W(t_{ij})=\langle e_2,e_3,e_6,e_7\rangle.
	\]
	The map $S\hookrightarrow \Xc_n$ is as follows:
	\[
	(U(t_{ij}),W(t_{ij}))\mapsto\begin{pmatrix}
		0 & 0 & t_{13} & 0 & 0 & t_{16} & t_{17}\\
		0 & 0 &0 & 0 & 0 &0 &0 \\
		0 & 0 &0 & 0 & 0 &0 &0 \\
		0 &0 &0 &0 &0 & t_{46} & t_{47}\\
		0 &0 &0 &0 &0 & t_{56} & t_{57}\\
		0 & 0 &0 & 0 & 0 &0 &0 \\
		0 & 0 &0 & 0 & 0 &0 &0 \\
	\end{pmatrix}.
	\]

\end{example}

\section*{Acknowledgements} I am grateful to Mikhail Ignatyev, Allen Knutson, Anna Melnikov and Dmitry Timashev for useful discussions on different stages of this project.


\begin{thebibliography}{MWZ99}

\bibitem[Ful97]{Fulton97}
William Fulton.
\newblock {\em Young tableaux}, volume~35 of {\em London Mathematical Society
  Student Texts}.
\newblock Cambridge University Press, Cambridge, 1997.
\newblock With applications to representation theory and geometry.

\bibitem[Ign12]{Ignatev12}
Mikhail~V. Ignatyev.
\newblock Combinatorics of {$B$}-orbits and {B}ruhat-{C}hevalley order on
  involutions.
\newblock {\em Transform. Groups}, 17(3):747--780, 2012.

\bibitem[KZJ14]{KnutsonZinn14}
Allen Knutson and Paul Zinn-Justin.
\newblock The {B}rauer loop scheme and orbital varieties.
\newblock {\em J. Geom. Phys.}, 78:80--110, 2014.

\bibitem[Lit94]{Lit94}
Peter Littelmann.
\newblock On spherical double cones.
\newblock {\em J. Algebra}, 166(1):142--157, 1994.

\bibitem[Man98]{Manivel98}
Laurent Manivel.
\newblock {\em Fonctions sym\'etriques, polyn\^omes de {S}chubert et lieux de
  d\'eg\'en\'erescence}, volume~3 of {\em Cours Sp\'ecialis\'es [Specialized
  Courses]}.
\newblock Soci\'et\'e Math\'ematique de France, Paris, 1998.

\bibitem[Mel00]{Mel00}
Anna Melnikov.
\newblock {$B$}-orbits in solutions to the equation $x^2=0$ in triangular
  matrices.
\newblock {\em Journal of Algebra}, 223(1):101--108, 2000.

\bibitem[Mel06]{Mel06}
Anna Melnikov.
\newblock Description of {$B$}-orbit closures of order 2 in upper-triangular
  matrices.
\newblock {\em Transformation groups}, 11(2):217--247, 2006.

\bibitem[MWZ99]{MWZ99}
Peter Magyar, Jerzy Weyman, and Andrei Zelevinsky.
\newblock Multiple flag varieties of finite type.
\newblock {\em Adv. Math.}, 141(1):97--118, 1999.

\bibitem[Smi08]{Smi08}
Evgeny Smirnov.
\newblock Resolutions of singularities for {S}chubert varieties in double
  {G}rassmannians.
\newblock {\em Funktsional. Anal. i Prilozhen.}, 42(2):56--67, 96, 2008.

\end{thebibliography}

\end{document}